\documentclass[12pt]{amsart}

\usepackage{amsmath}
\usepackage{amsfonts}
\usepackage{amssymb}
\usepackage{latexsym}
\usepackage{graphicx}
\usepackage{enumerate}
\usepackage{multirow}
\usepackage{subfigure}

\newtheorem{theorem}{Theorem}
\newtheorem{proposition}[theorem]{Proposition}
\newtheorem{lemma}[theorem]{Lemma}
\newtheorem{cor}[theorem]{Corollary}
\newtheorem{remark}{Remark}

\usepackage{mathtools}

\makeatletter
\DeclareRobustCommand\widecheck[1]{{\mathpalette\@widecheck{#1}}}
\def\@widecheck#1#2{%
    \setbox\z@\hbox{\m@th$#1#2$}%
    \setbox\tw@\hbox{\m@th$#1%
       \widehat{%
          \vrule\@width\z@\@height\ht\z@
          \vrule\@height\z@\@width\wd\z@}$}%
    \dp\tw@-\ht\z@
    \@tempdima\ht\z@ \advance\@tempdima2\ht\tw@ \divide\@tempdima\thr@@
    \setbox\tw@\hbox{%
       \raise\@tempdima\hbox{\scalebox{1}[-1]{\lower\@tempdima\box
\tw@}}}%
    {\ooalign{\box\tw@ \cr \box\z@}}}
\makeatother

\newcommand{\ep}{\epsilon}

\newcommand{\al}{\alpha}

\DeclareMathOperator{\sgn}{sgn}

\DeclareMathOperator{\sinc}{sinc}
\DeclareMathOperator{\Res}{Res}

\newcommand{\ds}{\displaystyle}

\newcommand{\be}{\begin{equation}}
\newcommand{\ee}{\end{equation}}
\newcommand{\bes}{\begin{equation*}}
\newcommand{\ees}{\end{equation*}}

\newcommand{\mand}{\quad \text{and}\quad}

\newcommand{\R}{{\bf{R}}}

\newcommand{\Z}{{\bf{Z}}}

\newcommand{\G}{{\mathcal{G}}}

\newcommand{\F}{{\mathfrak{F}}}

\renewcommand{\H}{{\mathcal{H}}}
\newcommand{\E}{{\mathcal{E}}}

\renewcommand{\tilde}{\widetilde}
\renewcommand{\hat}{\widehat}
\renewcommand{\check}{\widecheck}

\newcommand{\bunderbrace}[2]{%
  \begin{array}[t]{@{}c@{}}
  \underbrace{#1}\\
  #2
  \end{array}
}

\title{Approximation of Calogero-Moser lattices by Benjamin-Ono equations}
\author{J. Douglas Wright}

\begin{document}
\begin{abstract}We provide a rigorous validation that the infinite Calogero-Moser lattice can be well-approximated by solutions of the Benjamin-Ono equation in a long-wave limit.\end{abstract}
\maketitle

\section{Introduction}
The (generalized\footnote{It is {\it the} Calogero-Moser system when $\al=2$.}) Calogero-Moser system is
\be\label{CM}
\ddot{x}_j= -\alpha \sum_{m \ge 1} \left[  {1 \over \left(x_{j+m}-x_j\right)^{\alpha+1}} - {1 \over \left(x_j-x_{j-m} \right)^{\alpha+1}}\right].
\ee
In the above, $j \in \Z$, $x_j \in \R$, $t \in \R$.  The system can be interpreted as the governing equations for the positions ($x_j(t)$) of infinitely many particles arranged on a line and interacting pairwise through a power-law force.

Ingimarson \& Pego in \cite{IP} state that for $\alpha \in (1,3)$ and in a certain scaling regime (the so-called {\it long-wave limit}) the system is formally approximated by a 
Benjamin-Ono type equation. 
Here is a quick summary of their findings.
Suppose that $u=u(X,\tau)$ solves the (generalized\footnote{It is {\it the} Benjamin-Ono equation when $\al=2$.}) Benjamin-Ono equation
\be\label{BO}
\kappa_1 \partial_\tau u +\kappa_2 u \partial_X u + \kappa_3 H|D|^\alpha u=0.
\ee
Here $H$ is the Hilbert transform on $\R$ and $|D| = H\partial_X$. We define these as Fourier\footnote{We 
use the following form of the Fourier transform: $\ds \F[f](k) := \hat{f}(k):=(2 \pi)^{-1} \int_\R f(X) e^{-ikX} dX$
and $\F^{-1}[g](X)\ds:=\check{g}(X):=\int_\R g(k) e^{ikX} dk$.
We use the Fourier transform to  define Sobolev norms in the usual way:
$\| f \|_{H^s}:=\ds\sqrt{\int_\R (1+k^2)^s |\hat{f}(k)|^2 dk}$.} multiplier operators:
$$
\hat{H f}(k) := - i \sgn(k) \hat{f}(k) \mand \hat{|D|^\alpha f}(k) := |k|^\alpha \hat{f}(k). 
$$
The constants $\kappa_1$, $\kappa_2$ and $\kappa_3$ are determined from $\al$ by
\be\label{constants}\begin{split}
c_\alpha& := \sqrt{\alpha (\alpha+1) \zeta_\alpha},\quad \kappa_1 := 2 c_\alpha,\quad \kappa_2:=\al(\al+1)(\al+2) \zeta_\al\\
&\mand \kappa_3:=\al(\al+1) \int_0^\infty {1 -\sinc^2(s/2) \over s^\al} ds.
\end{split}\ee
Here $\ds \zeta_s: = \sum_{m \ge 1} 1/m^{s}$ is the ballyhooed zeta-function.

In \cite{IP} the authors show that if $u = -\partial_X v$ and 
\be\label{ansatz}
\tilde{x}_j(t) := j + \tilde{y}_j(t) \mand \tilde{y}_j(t):=\ep^{\al -2} v(\ep(j-c_\alpha t),\ep^{\al} t)
\ee
then
\bes\label{this is R}
R_\ep(j,t):=\ddot{\tilde{x}}_j+\alpha\sum_{m \ge 1} \left[ {1 \over (\tilde{x}_{j+m} - \tilde{x}_j)^{\alpha+1}} -{1 \over (\tilde{x}_{j} + \tilde{x}_{j-m})^{\alpha+1}} 
  \right]
\ees
is formally $o(\ep^{2\alpha-1})$ as $\ep \to 0^+$. We call $R_\ep$ {\it the residual} and it indicates the amount by which the approximation fails to satisfy \eqref{CM}. 
The scaling in \eqref{ansatz} is what is referred to as  the {\it  long-wave scaling}.

Our goal here is to provide a quantitative and rigorous error estimate
on the difference of true solutions of \eqref{CM}  and the approximate solution described in \cite{IP}.
Our main result is
\begin{theorem}\label{main}
There exists $\al_* \in (1.45,1.5)$ such that the following holds for $\al \in (\al_*,3)$.
Let $$\gamma_\alpha := \begin{cases} 
2\al-5/2, & \alpha \in (1,2] \\
3/2, & \alpha \in (2,3)
\end{cases}$$
and determine $c_\al$, $\kappa_1$, $\kappa_2$ and $\kappa_3$ as in \eqref{constants}.
Suppose that, for some $\tau_0>0$,  $u(X,\tau)$ solves \eqref{BO} for $|\tau| \le \tau_0$ and
$
\sup_{|\tau|\le \tau_0} \| u(\cdot,\tau) \|_{H^6} < \infty.
$ 
Then there exists $C_1,C_2,\ep_*>0$ so the following holds for $\ep \in (0,\ep_*]$.

If the initial data for \eqref{CM} satisfies
\bes
x_{j+1}(0)-x_j(0) - 1 =- \ep^{\al-1} u(\ep j, 0) + \bar{\mu}_j \mand
\dot{x}_j(0) = c_\al \ep^{\al-1} u(\ep j,0) + \bar{\nu}_j
\ees
where $$\|\bar{\mu}\|_{\ell^2} \le C_1 \ep^{\gamma_\al} \mand \|\bar{\nu}\|_{\ell^2} \le C_1 \ep^{\gamma_\al}$$
then the solution of \eqref{CM} satisifies
\bes\begin{split}
r_j(t)&:=x_{j+1}(t)-x_j(t) - 1 =- \ep^{\al-1} u(\ep (j-c_\al t),\ep^\al t)+ \mu_j(t) \\\mand
p_j(t)&:=\dot{x}_j(t) = c_\al \ep^{\al-1} u(\ep (j-c_\al t),\ep^\al t) + \nu_j(t)
\end{split}\ees
where
$$
\sup_{|t|\le \tau/\ep^\al} \|\mu(t)\|_{\ell^2} \le C_2 \ep^{\gamma_\al} \mand \sup_{|t|\le \tau/\ep^\al} \|{\nu}(t)\|_{\ell^2} \le C_2 \ep^{\gamma_\al}.
$$
\end{theorem}

\begin{remark} 
The theorem presents the absolute error made in the approximation. To compute the relative error we 
note that the long-wave scaling $X = \ep j$ implies
$\|\ep^{\al-1}u(\ep (\cdot-c_\al t),\ep^\al t)\|_{\ell^2} \le C \ep^{\al - 1/2}$ (see
estimate (4.8) from Lemma 4.3 in \cite{GMWZ}). This leads to a relative error like $C \ep^{\gamma_\al - \al + 1/2}=C\ep^{1-|\al -2|}$.
We do think the error estimates we compute here are sharp, though we do not have a proof of that.
\end{remark}

\begin{remark}
In our proof, it comes out that we need $2 \zeta_{\al+1}-\zeta_\al > 0$ and it is here that the restriction on $\al$ comes from. See Figure \ref{zetafig} below. We do not claim the condition is necessary, but it does arise in a somewhat natural way.
\end{remark}

\begin{remark}
The use of $H^6$ in the theorem is a worst-case scenario. It works for all $\al \in (1,3)$. If one wanted, one could determine a lower regularity condition on $u$ which would depend on $\al$. There is no pressing need for that in this article. One may wonder if $H^6$ solutions of \eqref{BO} exist. The short answer is yes. To get more information, the introduction of \cite{hur} gives a terrific overview.
\end{remark}

\begin{remark} For $\al =2$, there are known connections between special solutions of \eqref{CM}
and \eqref{BO}, which rely in part on the fact that both systems are integrable, 
see for instance \cite{matsuno}. These results are 
complementary to the discovery in \cite{IP} that the two systems are connected in the long-wave limit.
\end{remark}

\begin{remark} The Benjamin-Ono equation has served as  long-wave limit in a variety of hydrodynamic problems, see \cite{BLC} for an overview. 
The recent article \cite{IRTW} contains a rigorous validation of one such limit and is similar in spirit
to the result here.
\end{remark}

Here is the plan of attack.
First we make the formal estimates on $R_\ep$ from \cite{IP} rigorous in Section \ref{res est}.
Then we prove a general approximation theorem in Section \ref{gen section}. Lastly, in Section \ref{main thm section} we put things together in the proof of Theorem \ref{main}.

\section{Rigorous residual estimates}\label{res est}
The first task is to make the formal estimate of the residual $R_\ep$ from \cite{IP} rigorous. Here is the result:

\begin{proposition} \label{Res est 1}
If $u(X,\tau)$ is a solution of \eqref{BO} 
with $\sup_{|\tau| \le \tau_0} \|u(\cdot,\tau)\|_{H^6} < \infty$ then there exists $C>0$ and $\ep_0>0$ for which $\ep \in (0,\ep_0]$ implies
\bes\label{Res est 2}
\sup_{|t| \le \tau_0/\ep^{\al}} \| R_\ep(\cdot, t) \|_{\ell^2} \le C \ep^{\beta_\alpha}
\ees
where
$$
\beta_\alpha := \begin{cases} 
3\al-5/2, & \alpha \in (1,2] \\
\al+3/2, & \alpha \in (2,3).
\end{cases}
$$
\end{proposition}

\begin{proof} 
The proof is technical and we break it up into several parts: an analysis of the acceleration
term, another for the force terms and then a final section where we put everything together.

{\bf Part 1: the acceleration term.}
By the chain rule
$
\ddot{\tilde{x}}_j(t) =a_\ep(\ep(j-c_\al t),\ep^\al t)
$
where
$$
a_\ep(X,\tau) :=  -\ep^{\al} c_\al^2 \partial_X u(X,\tau) +  \ep^{2 \al-1} \kappa_1 \partial_\tau u(X,\tau) + \ep^{3 \al-2} \partial_\tau^2 v(X,\tau).
$$
Using \eqref{BO} we replace $\partial_\tau u$ to get:
$$
a_\ep= -\ep^{\al} c_\al^2 \partial_X u -  \ep^{2 \al-1} \kappa_2 u \partial_X u -  \ep^{2 \al-1}  \kappa_3 H|D|^\al u+ \ep^{3 \al-2} \partial_\tau^2 v.
$$
Differentiating \eqref{BO} with respect to $\tau$ and the relation $u = -\partial_X v$
imply\footnote{For concreteness, we note that we compute $v$ from $u$ via $\ds v(X,\tau) = -\int_0^X u(b,\tau)db$.}
$$
\partial_\tau^2 v = -{\kappa_2 \over \kappa_1}  u \partial_\tau u
+{\kappa_3 \over \kappa_1} |D|^{\al-1} \partial_\tau u.
$$
The Sobolev inequality and counting derivatives give 
$$
\|\partial_\tau^2 v\|_{H^1} \le C \|u\|_{H^1}\|\partial_\tau u\|_{H^1}
+ C \|\partial_\tau u\|_{H^{ \al}}.
$$
Taking the $H^s$ norm of both sides of \eqref{BO} tells us that $\| \partial_\tau u\|_{H^s} \le C\left(\|u\|^2_{H^{s+1}} + \| u\|_{H^{s+\al}}\right)$. In turn this gives
$$
\|\partial_\tau^2 v\|_{H^1} \le C \|u\|_{H^1}\left(\|u\|^2_{H^{2}} + \| u\|_{H^{1+\al}}\right)
+ C \left(\|u\|^2_{H^{\al+1}} + \| u\|_{H^{2\al}}\right).
$$
Since $\al < 3$ and we have assumed a uniform bound on $u \in H^6$ for $|\tau| \le \tau_0$ we can conclude
\be\label{final a est}
\sup_{|\tau| \le \tau_0} \| a_\ep+\ep^{\al} c_\al^2 \partial_X u +  \ep^{2 \al-1} \kappa_2 u \partial_X u +  \ep^{2 \al-1}  \kappa_3 H|D|^\al u\|_{H^1} \le C\ep^{3 \al -2}. 
\ee

{\bf Part 2: the force term.}
The authors of \cite{IP} show that 
$$
\tilde{x}_{j+m} - \tilde{x}_j = m - m\ep^{\al-1} A_{\ep m} u (X,\tau)
\mand \tilde{x}_{j} - \tilde{x}_{j-m} = m - m\ep^{\al-1} A_{-\ep m} u (X,\tau)
$$
where
$$
A_h u(X,\tau) := {1 \over h} \int_0^h u(X+z,\tau) dz.
$$
If we let
\bes\label{this is V}
V_m(g):={1 \over (m + g)^{\al}} - {1 \over m^\al} + {\alpha \over m^{\al+1}} g
\ees
so that
$$
V_m'(g) = -{\al \over (m+g)^{\al+1}} + {\al \over m^{\al+1}},
$$
then the force terms in $R_\ep$ can be rewritten as
\bes\begin{split}
\alpha\sum_{m \ge 1} \left[ {1 \over (\tilde{x}_{j+m} - \tilde{x}_j)^{\alpha+1}} -{1 \over (\tilde{x}_{j} + \tilde{x}_{j-m})^{\alpha+1}} 
 \right]
 =F_\ep(\ep(j-c_\al t),\ep^{\al} t)
 \end{split}\ees
 where
 $$
F_\ep(X,\tau):=-\sum_{m \ge 1} \left[ V'_m(- m\ep^{\al-1} A_{\ep m} u(X,\tau)) - V'_m(- m\ep^{\al-1} A_{-\ep m} u(X,\tau)) \right].
$$

A combination of the the fundamental theorem of calculus and Taylor's theorem implies
$$
V_m'(g_+) - V'_m(g_-) = 
V_m''(0)(g_+ - g_-) + {1 \over 2}V_m'''(0) (g_+^2 - g_-^2) + \int_{g_-}^{g_+} E_m(\sigma) d\sigma
$$
where
$$
E_m(\sigma) := \int_0^\sigma V_m''''(\phi)(\sigma-\phi) d\phi.
$$
This leads to the  expansion
\be\label{F decomp}\begin{split}
F_\ep = {\alpha (\alpha+1)}L_\ep +{\alpha (\alpha+1)(\alpha+2) \over 2} N_\ep +M_\ep
\end{split}\ee
where
\bes\begin{split}
L_\ep&:= \ep^{\al-1} \sum_{m \ge 1}  {1 \over m^{\alpha+1}} \left(A_{\ep m} - A_{-\ep m}\right) u\\
N_\ep&:= \ep^{2\al-2} \sum_{m \ge 1}  {1 \over m^{\alpha+1}} \left((A_{\ep m}u)^2 - (A_{-\ep m}u)^2\right) \\
M_\ep&:=-\sum_{m \ge 1} \int_{-m \ep^{\al-1} A_{-\ep m} u}^{-m \ep^{\al-1} A_{\ep m} u} E_m(\sigma) d\sigma.
\end{split}\ees
The terms $L_\ep$ and $N_\ep$ coincide with their forms in \cite{IP}. 
The remaining term, $M_\ep$, is lumped into a generic $o(\ep^{2\al+1})$ term there.

{\it Part 2a: Estimates for $L_\ep$.}
We follow the blueprint provided by \cite{IP}.
Using the fact that $\hat{A_h u}(k)  = \hat{u}(k){(e^{ikh}-1)}/ ikh$, they show that
\bes\begin{split}
\hat{L}_\ep(k) &= \ep^{\al}\zeta_\alpha  ik \hat{u}(k) + \ep^{2\al-1} ik |k|^{\alpha-1} \hat{u}(k) \left( |k| \ep \sum_{m \ge 1}{\sinc^2(|k|\ep m/2)-1\over (|k| \ep m)^{\alpha}}\right).
\end{split}\ees
Let
$$
\eta_\alpha( h) := h \sum_{m \ge 1}{1-\sinc^2(h m/2)\over (h m)^{\alpha}}.
$$
A key observation from \cite{IP} is that $\eta_\al(h)$ is the approximation of 
$$
\eta_\al:=\int_0^\infty {1 -\sinc^2(s/2) \over s^\al} ds
$$
using the rectangular rule with right hand endpoints. As such
$$
\lim_{h \to 0} \eta_\al(h) =\eta_\al.
$$
Note that $\eta_\al$ is finite so long as $\al \in (1,3)$.

Then we have
\be\label{L expando}
\hat{L}_\ep(k) = 
\ep^{\al}\zeta_\alpha  ik \hat{u}(k) - \ep^{2\al-1} \eta_\al ik |k|^{\alpha-1} \hat{u}(k) 
+\ep^{2\al-1}  ik |k|^{\alpha-1} \hat{u}(k) 
\left(\eta_\al- \eta_\al(\ep|k|)\right).
\ee
What is the error made by approximating $\eta_\al(\ep|k|)$ by $\eta_\al$? 
To determine this we 
need:
\begin{lemma}\label{eta lemma}
For $\al \in (1,2]$ there exists 
 $C>0$  for which
$|\eta_\al(h) - \eta_\al| \le C h$
for all $h > 0$. If $\al \in (2,3)$ there exists $C > 0$ for which 
$|\eta_\al(h) - \eta_\al| \le C h^{3-\al}$ for all $h > 0$.
\end{lemma}

\begin{proof} 
If the integral were not improper, this would be an elementary estimate. But it is.
In fact when $\al \in (2,3)$ it is improper at $s=0$
and that is why the estimate is worse in that setting. Also, when $\al \in (1,2)$ the derivative of 
the integrand
$$f_\al(s) :=  {1 -\sinc^2(s/2) \over s^\al}$$
diverges as $s \to 0^+$, which complicates things.

First we deal with $h \ge 1$. We have
$$
|\eta_\al(h) - \eta_\al|  \le \eta_\al + h \sum_{m \ge 1} {1 - \sinc^2(hm/2) \over (h m)^\al}.
$$
Since $\sinc^2(s) \in [0,1]$ for all $s\in \R$ we make an easy estimate
$$
|\eta_\al(h) - \eta_\al|  \le \eta_\al + h \sum_{m \ge 1} {1 \over (h m)^\al} = \eta_\al+ h^{1-\al}\zeta_\al 
= \left(\eta_\al + h^{-\al} \zeta_\al\right)h  \le (\eta_\al + \zeta_\al) h \le C h.
$$
So $h \ge 1$ is taken care of for all $\al \in (1,3)$.

Now fix $h \in (0,1)$. 
We break things up:
\bes\begin{split}
\eta_\al(h) - \eta_\al =  & \bunderbrace{h \sum_{m = 1}^{\lceil 1/h \rceil}f_\al(mh)- \int_0^{h \lceil 1/h \rceil} f_\al(s) ds}{IN}
+ \bunderbrace{h \sum_{m \ge  \lceil 1/h \rceil+1}f_\al(mh)- \int^\infty_{h \lceil 1/h \rceil} f_\al(s) ds}{OUT}.
\end{split}\ees

For $OUT$,  by standard integral identities and the integral version of the mean value theorem
we have
$$
OUT = \sum_{m \ge  \lceil 1/h \rceil+1} \left( hf_\al(mh) - \int_{(m-1)h}^{mh} f_\al(s) ds\right)
= h \sum_{m \ge  \lceil 1/h \rceil+1} \left(f_\al(mh) - f_\al(s_m)\right).
$$
Here $s_m \in [(m-1)h,mh]$. Then we use the derivative version of the mean value theorem to get
$$
OUT =  h \sum_{m \ge  \lceil 1/h \rceil+1} f'_\al(\sigma_m)(mh-s_m)
$$
where $\sigma_m \in [s_m,m h]$. Note that $|mh-s_m|\le h$. 

Routine calculations show that there is a constant $C>0$ such that 
$|f'_\al(s)| \le C s^{-\al-1}$ for $s \ge 1$. Since $s^{-\al-1}$ is a decreasing function these considerations lead to
$$
|OUT| \le C h^2 \sum_{m \ge  \lceil 1/h \rceil+1}  [(m-1) h]^{-\al-1} =C h^2 \sum_{m \ge  \lceil 1/h \rceil}  [m h]^{-\al-1}.
$$

Next, 
$\ds h \sum_{m \ge  \lceil 1/h \rceil}  [m h]^{-\al-1}$
is the approximation of $\ds \int^\infty_{h\lceil 1/h\rceil-h} s^{-\alpha-1} ds$ using the rectangular rule with right hand endpoints. Since $s^{-\al-1}$ is decreasing we know that $\ds h \sum_{m \ge  \lceil 1/h \rceil}  [m h]^{-\al-1} \le \int^\infty_{h\lceil 1/h\rceil-h} s^{-\alpha-1}ds$.
Also since $h \in (0,1)$ we have $h\lceil 1/h\rceil-h \ge 1/2$ and so 
$\ds \int^\infty_{h\lceil 1/h\rceil-h} s^{-\alpha-1} \le\int^\infty_{1/2} s^{-\alpha-1} = 2^\al/\al$. Putting these together imply $|OUT| \le Ch$.

For $IN$ we need to desingularize the integrand at $s = 0$. Putting
$$
f_\al(s) = \bunderbrace{{1-(s^2/12) -\sinc^2(s/2) \over s^\al}}{g_\al(s)} + {1 \over 12} s^{2 - \al}
$$
gives
$$
IN = \left(h \sum_{m = 1}^{\lceil 1/h \rceil}g_\al(mh)- \int_0^{h \lceil 1/h \rceil} g_\al(s) ds\right)
+{1 \over 12} \left(
h \sum_{m = 1}^{\lceil 1/h \rceil}(mh)^{2-\al} - \int_0^{h \lceil 1/h \rceil} s^{2-\al} ds\right).
$$
Taylor's theorem tells us that 
$|1-(s^2/12) -\sinc^2(s/2)|/s^4$ is bounded as $s \to 0$ and as a byproduct we see that $g_\al(s)$ is $C^1$ on 
the interval $[0,2]$.
Routine error estimates for approximating integrals with rectangles tells us 
 $$
 \left|h \sum_{m = 1}^{\lceil 1/h \rceil}g_\al(mh)- \int_0^{h \lceil 1/h \rceil} g_\al(s) ds\right| \le Ch.
$$

So what remains is to estimate the singular piece
$$
SING_\al = \left |h \sum_{m = 1}^{\lceil 1/h \rceil}(mh)^{2-\al} - \int_0^{h \lceil 1/h \rceil} s^{2-\al} ds\right|.
$$
Note that if $\al  = 2$ then $SING_\al = 0$, so that case is pretty easy. But the cases $\al \in (1,2)$ and $\al \in (2,3)$
require some care. 

We know that 
$\ds h \sum_{m =1}^{\lceil 1/h \rceil} {(mh)^{2-\al}}$ is the  rectangular approximation of $\ds \int_0^{h\lceil 1/h \rceil}  {s^{2-\al}} ds$ using right hand endpoints but it is also the rectangular approximation of $\ds \int_0^{h\lceil 1/h \rceil}  {(s+h)^{2-\al}} ds$ using left hand endpoints. 
If $\alpha \in (2,3)$ then $s^{2-\alpha}$ is a decreasing function and we get the following chain of inequalities:
\bes\label{big chain}
\int_0^{h\lceil 1/h \rceil}  {(s+h)^{2-\al}} ds \le h \sum_{m =1}^{\lceil 1/h \rceil} (mh)^{2-\al} \le \int_0^{h\lceil 1/h \rceil}  {s^{2-\al}} ds.
\ees
On the other hand if $\al \in (1,2)$ then $s^{2-\alpha}$ is increasing and we have
\bes\label{small chain}
\int_0^{h\lceil 1/h \rceil}  {(s+h)^{2-\al}} ds \ge h \sum_{m =1}^{\lceil 1/h \rceil} (mh)^{2-\al} \ge \int_0^{h\lceil 1/h \rceil}  {s^{2-\al}} ds.
\ees
Either of the chains tells us:
\bes\begin{split}
SING_\al &\le \left \vert
\int_0^{h\lceil 1/h \rceil}  \left(s^{2-\al} - {(s+h)^{2-\al}} \right)ds
\right \vert\\&={1 \over 3-\al}
\left| (h\lceil 1/h \rceil)^{3-\al}  -
(h\lceil 1/h \rceil+h)^{3-\al}  + h^{3-\al}\right|\\
&\le {1 \over 3-\al}
\left| (h\lceil 1/h \rceil)^{3-\al}  -
(h\lceil 1/h \rceil+h)^{3-\al}   \right| + {1 \over 3-\al} h^{3-\al}.
\end{split}\ees
The mean value theorem gives
$${1 \over 3-\al}
\left| (h\lceil 1/h \rceil)^{3-\al}  -
(h\lceil 1/h \rceil+h)^{3-\al}   \right|= {h h_*^{2-\al}}
$$
where $h_*$ is in between $h\lceil 1/h \rceil$ and $h\lceil 1/h \rceil+h$. These numbers are in the interval $[1,3]$ and so we have
$${1 \over 3-\al}
\left| (h\lceil 1/h \rceil)^{3-\al}  -
(h\lceil 1/h \rceil+h)^{3-\al}   \right| \le C h.
$$
So $|SING_\al| \le Ch + Ch^{3-\al}$.

Everything all together tells us that $h \in (0,1)$ and $\al \in (1,3)$ imply
$
|\eta_\al(h) - \eta_\al| \le Ch + C h^{3-\al}.
$
If $\al \in (1,2]$ then $h \le  h^{3-\al} $ and the inequality flips for $\al \in (2,3)$. That finishes the proof.

 \end{proof}

With Lemma \ref{eta lemma}, \eqref{L expando} implies
$$
\left \vert \hat{L}_\ep(k) - \ep^{\al}\zeta_\alpha  ik \hat{u}(k) + \ep^{2\al-1}\eta_\al ik |k|^{\alpha-1} \hat{u}(k)\right\vert
\le C \ep^{2\al-1+r_\al}|k|^{\al+r_\al} |\hat{u}(k)| 
$$
where $$
r_\al:=\begin{cases} 1, & \al \in (1,2]\\ 3-\al, & \al \in (2,3).
\end{cases}
$$ 
This in turn implies (along with the assumed uniform estimate for $u$) that
\be\label{L est}
\sup_{|\tau| \le \tau_0} 
\| L_\ep - \ep^{\al} \zeta_\alpha \partial_X u -\ep^{2\al-1} \eta_\al H |D|^{\al}u \|_{H^1} \le C \ep^{2\al-1+r_\al}.
\ee

{\it Part 2b: Estimates for $N_\ep$.} Some easy algebra leads to
\bes\begin{split}
N_\ep = & 2\ep^{2\al-2} u \sum_{m \ge 1} {1 \over m^{\al +1}}
(A_{\ep m} u - A_{-\ep m}u)\\
+&\ep^{2\al-2} \sum_{m\ge1} {1 \over m^{\al +1}}
(A_{\ep m} u + A_{-\ep m} u-2u)(A_{\ep m} u - A_{-\ep m}u).
\end{split}\ees
We recognize that $L_\ep$ is lurking in the first term and get
\bes\begin{split}
N_\ep = & 2\ep^{\al-1} u L_\ep
+\ep^{2\al-2} \sum_{m\ge1} {1 \over m^{\al +1}}
(A_{\ep m} u + A_{-\ep m} u-2u)(A_{\ep m} u - A_{-\ep m}u).
\end{split}\ees
Then we subtract:
\bes\begin{split}
N_\ep -2\ep^{2\al-1} \zeta_\al u \partial_X u =& 2 \ep^{\al-1} u \left(L_\ep - \ep^{\al} \zeta_\al \partial_X u\right)\\
+&\ep^{2\al-2} \sum_{m\ge1} {1 \over m^{\al +1}}
(A_{\ep m} u + A_{-\ep m} u-2u)(A_{\ep m} u - A_{-\ep m}u).
\end{split}\ees
We take the $H^1$ norm and use triangle and Sobolev:
\be\begin{split}\label{gack}
&\| N_\ep -2\ep^{2\al-1} \zeta_\al u \partial_X u\|_{H^1}\\ \le & 2 \ep^{\al-1} \|u\|_{H^1} \|L_\ep - \ep^{\al} \zeta_\al \partial_X u\|_{H^1}\\
+&\ep^{2\al-2} \sum_{m\ge1} {1 \over m^{\al +1}}
\|A_{\ep m} u + A_{-\ep m} u-2u\|_{H^1}\|A_{\ep m} u - A_{-\ep m}u\|_{H^1}.
\end{split}\ee
The estimate \eqref{L est} tells us that $2 \ep^{\al-1} \|u\|_{H^1} \|L_\ep - \ep^{\al} \zeta_\al \partial_X u\|_{H^1}
\le C\ep^{3\al-2} \|u\|^2_{1+\al + r_\al}$.

To control the remaining term in \eqref{gack} we will use the following estimates (see \cite{HML} for the proof):
\begin{lemma}\label{A est} There is $C>0$ such that for all $h>0$ and $s \in \R$:
\bes\begin{split}
\|A_h u\|_{H^s} &\le C \|u\|_{H^s}\\
\|A_{h} u + A_{-h} u-2u\|_{H^s}& \le C h^2 \|u\|_{H^{s+2}}\\
\|A_{h} u - A_{-h} u\|_{H^s} &\le C h \|u\|_{H^{s+1}}\\
\|A_{h} u -  u\|_{H^s} &\le C h \|u\|_{H^{s+1}}.
\end{split}\ees
\end{lemma}
We have to deal with terms like $A_{\ep m}$ and so the above result will be helpful
when $\ep m$ is ``small'' but not very useful otherwise. So we break things up:
\bes\label{drr}\begin{split}
&\ep^{2\al-2} \sum_{m\ge1} {1 \over m^{\al +1}}
\|A_{\ep m} u + A_{-\ep m} u-2u\|_{H^1}\|A_{\ep m} u - A_{-\ep m}u\|_{H^1}\\
=&\ep^{2\al-2} \sum_{m=1}^{\lfloor 1/\ep \rfloor} {1 \over m^{\al +1}}
\|A_{\ep m} u + A_{-\ep m} u-2u\|_{H^1}\|A_{\ep m} u - A_{-\ep m}u\|_{H^1}\\
+&\ep^{2\al-2} \sum_{m\ge \lfloor 1/\ep \rfloor} {1 \over m^{\al +1}}
\|A_{\ep m} u + A_{-\ep m} u-2u\|_{H^1}\|A_{\ep m} u - A_{-\ep m}u\|_{H^1}\\
=& I+II.
\end{split}\ees
Applying the second and third estimates from Lemma \ref{A est} gives
\bes\begin{split}
I
\le &C\ep^{2\al+1}  \|u\|^2_{H^{3}}\sum_{m=1}^{\lfloor 1/\ep \rfloor} {m^{2-\al}}.
\end{split}\ees
A classic ``integral comparison'' tells us that
$\ds
\sum_{m=1}^{\lfloor 1/\ep \rfloor} {m^{2-\al}} \le C \ep^{\al-3}.$
So then
$
I\le C \ep^{3\al-2}  \|u\|^2_{H^{3}}.
$

For $II$ we use the first estimate in Lemma \ref{A est} to get 
\bes\begin{split}
II
\le &C \ep^{2\al-2} \|u\|^2_{H^{1}}\sum_{m>\lfloor 1/\ep \rfloor} {1 \over m^{\al +1}}.
\end{split}\ees
Then another integral type estimate tells us $\ds
\sum_{m>\lfloor 1/\ep \rfloor} {1 \over m^{\al +1}} \le C \ep^{\al}.
$
So that $II \le C \ep^{3 \al-2}\|u\|_{H^1}^2$. 
Therefore we have our final estimate for $N_\ep$:
\be\label{final N est}\begin{split}
\sup_{|\tau| \le \tau_0} \| N_\ep -2\ep^{2\al-1} \zeta_\al u \partial_X u\|_{H^1} \le 
C \ep^{3\al-2}.
\end{split}\ee

{\it Part 2c: Estimates for $M_\ep$.} We need to treat $\|M_\ep\|_{L^2}$ and $\| \partial_X M_\ep\|_{L^2}$ separately and we start with the former. An routine estimate shows
$$
|M_\ep| \le \sum_{m \ge 1} m \ep^{\al-1}|(A_{\ep m} - A_{-\ep m})u| \sup_{\sigma \in I_m} 
|E_m(\sigma)|
$$
where $I_m$ is the interval between $-m \ep^{\al-1} A_{\ep m} u$ and $-m \ep^{\al-1} A_{-\ep m} u$. 

If we assume $\sigma>0$ then
\bes\begin{split}
|E_m(\sigma)| 
\le & \int_0^\sigma |V_m''''(\phi)| |\sigma-\phi| d \phi 
\le  \sup_{0\le \phi\le \sigma }|V_m''''(\phi)|  \int_0^\sigma |\sigma-\phi| d \phi 
=  {1 \over 2} \sup_{0\le \phi\le \sigma }|V_m''''(\phi)| \sigma^2.
\end{split}\ees
$|V_m''''(\phi)|$ is decreasing and so $\sup_{0\le \phi\le \sigma }|V_m''''(\phi)| = |V_m''''(0)| = C/m^{\al+4}$.
So in this case
$$
|E_m(\sigma)| \le {C \sigma^2\over m^{\al+4}} .
$$
Similarly if $\sigma \le 0$:
\bes\begin{split}
|E_m(\sigma)| 
\le  \int_\sigma^0 |V_m''''(\phi)| |\sigma-\phi| d \phi 
\le  {1 \over 2} \sup_{0\le \phi\le \sigma }|V_m''''(\phi)| \sigma^2
\le  {C \sigma^2\over (m+\sigma)^{\al+4}} .
\end{split}\ees
In either case we had:
\be\label{Em est}
|E_m(\sigma)| \le {C \sigma^2\over (m-|\sigma|)^{\al+4}} .
\ee
Note that the constant $C$ here is independent of $m$.

We have, using Lemma \ref{A est} and Sobolev:
\be\label{sig est 1}
\sigma \in I_m \implies |\sigma| \le C m \ep^{\al-1}\|u\|_{H^1}.
\ee
In particular by ensuring that $\ep$ is not so large we have
\be\label{sig est 2}
\sigma \in I_m \implies |\sigma| \le m/2.
\ee
So  \eqref{Em est}, \eqref{sig est 1} and \eqref{sig est 2} give
\be\label{Em est 2}
\sup_{\sigma\in I_m} |E_m(\sigma)|\le \left({C m^2 \ep^{2\al-2} \|u\|_{H^1}^2\over (m-m/2)^{\alpha+4}}\right) \le {C \ep^{2\al-2} \over m^{\al +2}} \|u\|_{H^1}^2.
\ee
In turn this gives
$$
|M_\ep| \le C\ep^{3\al-3} \|u\|_{H^1}^2\sum_{m \ge 1} {1 \over m^{\al +1}}|(A_{\ep m} - A_{-\ep m})u|. 
$$
Then Lemma \ref{A est} leads us to:
\bes\label{M L2}\begin{split}
\|M_\ep\|_{L^2} &
\le C\ep^{3\al-3} \|u\|_{H^1}^2\sum_{m \ge 1} {1 \over m^{\al +1}}\|(A_{\ep m} - A_{-\ep m})u\|_{L^2} \\
&\le C\ep^{3\al-2} \|u\|_{H^1}^3\sum_{m \ge 1} {1 \over m^{\al}}\\
&\le C\zeta_\al\ep^{3\al-2} \|u\|_{H^1}^3.
\end{split}\ees

Next we compute using the fundamental theorem and some algebra:
$$
\partial_X M_\ep = \ep^{\al-1}
\sum_{m \ge 1} m \left[
E_m(-m \ep^{\al-1} A_{\ep m} u) A_{\ep m} \partial_X u 
-E_m(-m \ep^{\al-1} A_{-\ep m} u)   A_{-\ep m} \partial_X u \right].
$$
Adding zero takes us to
\bes\begin{split}
\partial_X M_\ep &= \ep^{\al-1}
\sum_{m \ge 1} m 
E_m(-m \ep^{\al-1} A_{\ep m} u) \left(A_{\ep m} 
-  A_{-\ep m}  \right)\partial_X u\\
&+
 \ep^{\al-1}
\sum_{m \ge 1} m \left(
E_m(-m \ep^{\al-1} A_{\ep m} u)-
E_m(-m \ep^{\al-1} A_{-\ep m} u)\right)   A_{-\ep m} \partial_X u \\
&=III+IV.
\end{split}\ees
Using \eqref{Em est} and the same reasoning that lead to \eqref{sig est 2} yields
\bes\begin{split}
|III|
\le &C \ep^{\al-1} 
\sum_{m \ge 1} m {\left( m \ep^{\al-1} A_{\ep m} u  \right)^2 \over m^{\al+4}}
\left| (A_{\ep m}-A_{\ep m}) \partial_Xu\right|\\
\le &C \ep^{3\al -3}
\sum_{m \ge 1} {1 \over m^{\alpha+1}}
 |A_{\ep m} u|^2 \left| (A_{\ep m}-A_{\ep m}) \partial_Xu\right|.
\end{split}\ees
Sobolev and the first estimate in Lemma \ref{A est} imply 
$ |A_{\ep m} u(X)| \le C \|u\|_{H^1}$ and so
\bes\begin{split}
|III|
\le C \ep^{3\al -3} \|u\|_{H^1}^2 
\sum_{m \ge 1} {1 \over m^{\alpha+1}}
\left| (A_{\ep m}-A_{\ep m}) \partial_Xu\right|.
\end{split}\ees
We take the $L^2$-norm of the above, use the second estimate in Lemma \ref{A est} and do the resulting sum to obtain
\bes\begin{split}
\|III\|_{L^2}
\le C \ep^{3\al -3} \|u\|_{H^1}^2 
\sum_{m \ge 1} {1 \over m^{\alpha+1}} \ep m \|\partial_Xu\|_{H^1}
\le C \ep^{3\al-2} \|u\|^{3}_{H^2}.
\end{split}\ees

For $IV$, routine estimates and the mean value theorem give
$$
|IV|  \le \ep^{2\al-2}
\sum_{m\ge1} m^2 |A_{-\ep m} \partial_X u| 
|(A_{\ep m}-A_{-\ep m}) u| \sup_{\sigma \in I_m} |E'_m(\sigma)|
$$
where $I_m$ is consistent with its definition above. Reasoning analogous to that which led 
\eqref{Em est} can be used to show that 
$|E'_m(\sigma)|\le C|\sigma|/(m-|\sigma|)^{\al + 4}.$ And then \eqref{sig est 1} and \eqref{sig est 2} imply
$$
\sup_{\sigma \in I_m} |E_m'(\sigma)| \le {C \ep^{\al-1} \over m^{\al + 3}} \|u\|_{H^1}.
$$
So 
$$
|IV| \le C \ep^{3\al -3} \|u\|_{H^1} \sum_{m \ge 1} {1 \over 
 m^{\al +1}} |A_{-\ep m} \partial_X u| 
|(A_{\ep m}-A_{-\ep m}) u|. 
$$

Using Sobolev and the estimates in Lemma \ref{A est} in ways we have done above
give
$
\|IV\|_{L^2} \le  C \ep^{3\al -2} \|u\|_{H^1}^3. 
$
This, in conjunction with the above estimates for $\|III\|_{L^2}$ and $\|M_\ep\|_{L^2}$ result in
\be\label{final M est}
\sup_{|\tau| \le \tau_0} \| M_\ep \|_{H^1} \le C \ep^{3 \al -2}.
\ee

{\bf Part 3: finishing touches.} Using the decomposition of $F_\ep$ from \eqref{F decomp}
along with the definitions of $c_\al$, $\kappa_1$, $\kappa_2$ and $\kappa_3$ we obtain
after some routine triangle inequality estimates
\bes\begin{split}
\|a_\ep + F_\ep\|_{H^1}  = &
\left\|a_\ep + \al(\al+1) L_\ep + {\al(\al+1)(\al+2)\over 2} N_\ep + M_\ep \right\|_{H^1} \\
\le &\left\|a_\ep +\ep^{\al} c_\al^2 \partial_X u +  \ep^{2 \al-1} \kappa_2 u \partial_X u +  \ep^{2 \al-1}  \kappa_3 H|D|^\al u\right\|_{H^1}\\
+& \al(\al+1)\left\| L_\ep -\ep^{\al} \zeta_\al \partial_X u -\ep^{2 \al-1}  \eta_\al H|D|^\al u\right\|_{H^1}\\
+ &{\al(\al+1)(\al+2)\over 2}  \left\|N_\ep- 2 \ep^{2 \al-1}\zeta_\al  u \partial_X u \right\|_{H^1}\\
+&\left\| M_\ep \right\|_{H^1}.
\end{split}\ees
So we use \eqref{final a est}, \eqref{L est}, \eqref{final N est} and \eqref{final M est} to get
$$
\sup_{|\tau| \le \tau_0} \|a_\ep + F_\ep\|_{H^1} \le
C \left( \ep^{2\al - 1 + r_\al}+\ep^{3 \al -2} \right).
$$
For $\al \in (1,2]$ we have $\ep^{2 \al -1 + r_\al} = \ep^{2 \al} \le \ep^{3 \al - 2}$
and so
\bes\label{rest1}\al \in (1,2] \implies
\sup_{|\tau| \le \tau_0} \|a_\ep + F_\ep\|_{H^1} \le
C \ep^{3 \al -2}.
\ees
For $\al \in (2,3)$ we have $\ep^{2 \al -1 + r_\al} = \ep^{\al+2} \ge \ep^{3 \al - 2}$ so
\bes\label{rest2}\al \in (1,2] \implies
\sup_{|\tau| \le \tau_0} \|a_\ep + F_\ep\|_{H^1} \le
C \ep^{\al+2}.
\ees
By design, we have $R_\ep(j,t) = a_\ep(\ep(j-c_\al t),\ep^\al t) + F_\ep (\ep(j-c_\al t),\ep^\al t)$.
 Estimate (4.8) from Lemma 4.3 in \cite{GMWZ} states that if 
\be\label{lw sample}
 G(X) \in H^1 \implies  \|G(\ep \cdot)\|_{\ell^2} \le C \ep^{-1/2} \|G\|_{H^1}.
\ee 
Thus \eqref{rest1} and \eqref{rest2} give us
\bes\label{Res est 3}
\sup_{|t| \le \tau_0/\ep^{\al}} \| R_\ep(\ep(\cdot - t),\ep^{\al} t) \|_{\ell^2} \le C \ep^{\beta_\alpha}
\ees
with $\beta_\al$ as in the statement of the proposition. That does it.
\end{proof}

\section{A general approximation theorem}\label{gen section}

The previous section provides a rigorous bound on the size of the residual $R_\ep$, but this is only part of the approximation theory. We need to demonstrate that a true solution $x_j(t)$ is shadowed by the approximate solution $\tilde{x}_j(t)$ in an appropriate sense. The argument (which is a direct descendent of the validation of KdV as the long-wave limit for FPUT lattices in \cite{SW}) is based on ``energy estimates''. This estimates are much more transparent after a change of coordinates. After the recoordinatization, we prove a conservation law which will imply global in time existence of solutions. Then we prove a general approximation result.

\subsection{Relative displacement/velocity coordinates.}
Let
$$
r_j:= x_{j+1} - x_j-1 \mand p_j:=\dot{x}_j.
$$
These are referred to as the relative displacement and velocity. 
Note that if $r_j = p_j=0$ then the system is in the equilibrium configuration $x_j = j$.
A calculation shows that
%
%
%
%
%
$$
x_{j+m} - x_j =m+ \G_m r_j
\mand 
x_j - x_{j-m} =m+\G_{-m} r_j
$$
where
$$
\G_m r_j:=\sum_{l = 0}^{m-1} r_{j+l}  \mand \G_{-m} r_j:=\sum_{l=0}^{m-1} r_{k-m-l}.
$$
Also define operators $S^k$, $\delta^\pm_m$ via:
$$
S^k f_j := f_{j+k},\quad \delta^+_m f_j:=f_{j+m}-f_j\mand\delta^-_m f_j := f_j-f_{j-m}.
$$
Here are a few useful formulas that are not too hard to confirm:
$$
\G_m \delta_1^+ = \delta_m^+ \mand S^{-m} \G_m r_j = \G_m r_{j-m} = \G_{-m} r_j.
$$

The above considerations allow us to reformulate \eqref{CM} as a first order system in terms of $r_j$ and $p_j$:
\be\label{CM3}\begin{split}
\dot{r}_j&= \delta^+_1 p_j  \mand
\dot{p}_j= \sum_{m \ge 1} \delta^-_m V_m'(\G_m r)_j
\end{split}\ee
where $V_m$ coincides with its definition in the previous section.

\subsection{Energy conservation}
Let
$$
\E:=K+P
$$ 
where
$$
K:={1 \over 2} \sum_{j \in \Z}p_j^2
\mand
P:=
\sum_{j\in \Z} \sum_{m \ge 1} V_m(\G_m r)_j.
$$
This quantity corresponds to the total mechanical energy of the system. It is constant and here is the extremely classical argument.

Differentiation of $\E$ with respect to $t$:
\bes\begin{split}
\dot{\E} = \sum_{j \in \Z} \left( p_j \dot{p}_j + \sum_{m \ge 1} V_m'(\G_m r)_j \G_m \dot{r}_j \right).
\end{split}\ees
Eliminate $\dot{p}$ and $\dot{r}$ on the right using \eqref{CM3}:
\bes\begin{split}
\dot{\E} = \sum_{j \in \Z} \left( p_j  \sum_{m \ge 1} \delta^-_m V_m'(\G_m r)_j + \sum_{m \ge 1} V_m'(\G_m r)_j \G_m \delta^+_1 p_j \right).
\end{split}\ees
Rearrange the sums:
\bes\begin{split}
\dot{\E} =  \sum_{m \ge 1}\sum_{j \in \Z} \left( p_j  \delta^-_m V_m'(\G_m r)_j + V_m'(\G_m r)_j \G_m \delta^+_1 p_j \right).
\end{split}\ees
Use $\G_m \delta_1^+ = \delta_m^+$ in the second term and the summation by parts formula in the first:
\bes\begin{split}
\dot{\E} =  \sum_{m \ge 1}\sum_{j \in \Z} \left(  -\delta^+_m p_j V_m'(\G_m r)_j + V_m'(\G_m r)_j \delta^+_m p_j \right)=0.
\end{split}\ees
So $\E$ is constant.

\subsection{Norm equivalence and global existence}
In certain circumstances, the (square root of  the) energy $E$ from the previous section is equivalent to the $\ell^2 \times \ell^2$ norm. It is trivial that $K$ is equivalent to $\| p \|_{\ell^2}^2$ but the part involving $P$ is not obvious. Here is the result.
\begin{lemma}\label{P equivalence}
Suppose that $\alpha >1$ such that $2\zeta_{\alpha+1} - \zeta_\alpha >0$. Then there exists $\rho>0$ and a constant $C >1$ such that 
$$
\| r \|_{\ell^2} \le \rho \implies C^{-1} \|r\|_{\ell^2} \le \sqrt{P} \le C \|r\|_{\ell^2}.
$$
\end{lemma}
The proof of this will be a consequence of Proposition \ref{W prop}, below. For now, note that it implies that small initial data for \eqref{CM3} implies global existence of solutions. Specifically:
\begin{cor} Fix $\alpha>1$ such that $2\zeta_{\alpha+1} - \zeta_\alpha >0$. Then there exists $\rho,C>0$
such that if $\|\bar{r},\bar{p}\|_{\ell^2 \times \ell^2} \le \rho$ then there exists $(r_j(t),p_j(t)) \in C^1(\R;\ell^2 \times \ell^2)$ which solves \eqref{CM3} and for which $(r_j(0),p_j(0)) = (\bar{r}_j,\bar{p}_j)$. Moreover
$\sup_{t \in\R} \|r(t),p(t)\|_{\ell^2 \times \ell^2} \le C \|\bar{r},\bar{p}\|_{\ell^2 \times \ell^2}$.
\end{cor}
We omit the proof as it is classical. In any case, it is nearly identical to the proof of Theorem 5.2 of \cite{GMWZ}.

\begin{remark} It is non-obvious that the condition $2\zeta_{\alpha+1} -\zeta_\alpha>0$ is met. In
Figure \ref{zetafig} we plot $2\zeta_{\alpha+1} -\zeta_\alpha$ vs $\al$. One sees that 
there exists a root of $2\zeta_{\alpha+1} -\zeta_\alpha$,
 denoted $\al_*$, in the interval $(1.4,1.5)$, such that
 $2\zeta_{\alpha+1} -\zeta_\alpha>0$ for $\al > \al_*$ and is non-positive otherwise. 
Thus the condition is non-vacuous.
 \begin{figure}[t]\centering
\includegraphics[width=.75\textwidth]{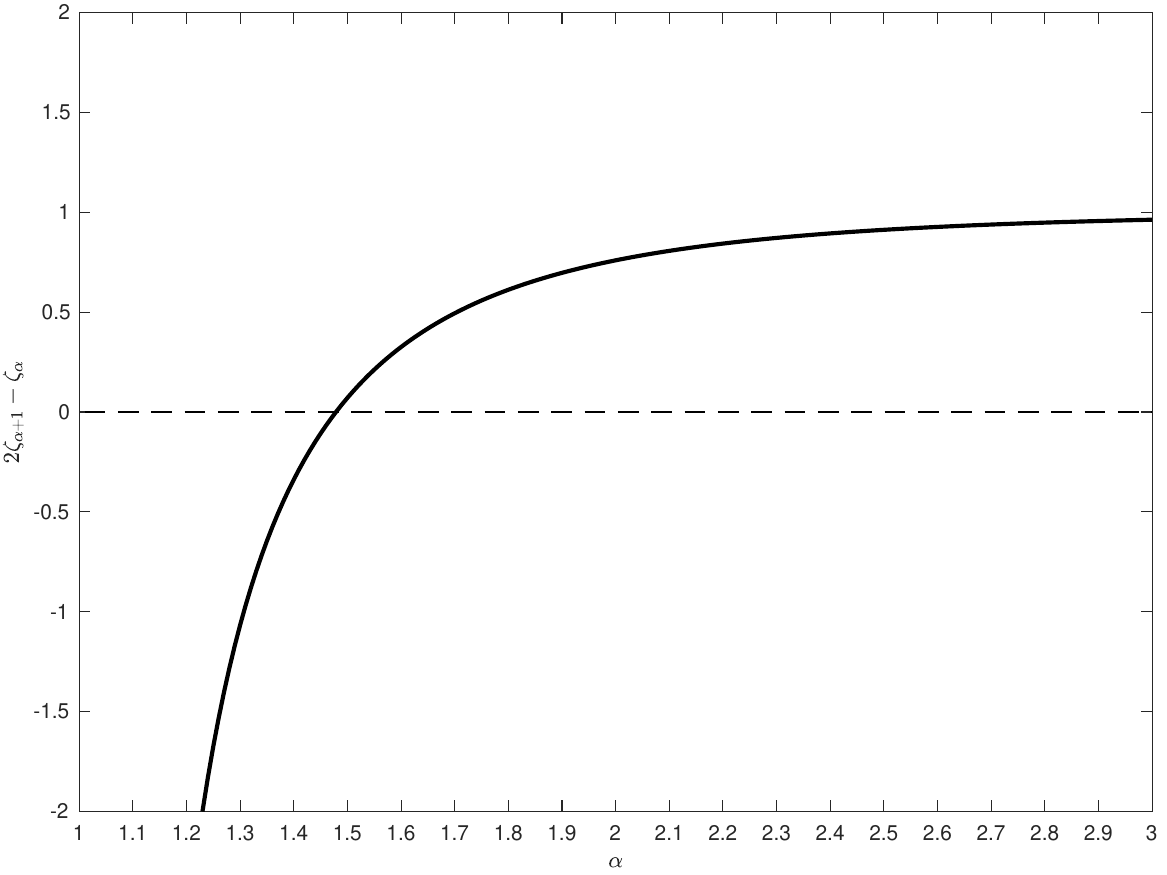}
\caption{$2\zeta_{\alpha+1} -\zeta_\alpha$ vs $\alpha$}\label{zetafig}
\end{figure}
\end{remark}
%
%
%
%
\subsection{Approximation in general}
Let 
$$
r_j(t) = \tilde{r}_j(t) + \eta_j(t) \mand p_j(t) = \tilde{p}_j(t) + \xi_j(t).
$$
where $\tilde{r}_j(t)$ and $\tilde{p}_j(t)$ are some given functions which we expect are good approximators
to true solutions $r_j(t)$ and $p_j(t)$ of \eqref{CM3}. Then the ``errors'' $\eta_j(t)$ and $\xi_j(t)$ solve
\be\label{err eqn}
\dot{\eta}_j = \delta_1^+ \xi_j + \Res_1 \mand \dot{\xi}_j = \sum_{m \ge 1} \delta_m^- \left[V_m'\left( \G_m \left(\tilde{r} +\eta\right)\right) 
-V_m'\left( \G_m\tilde{r}\right) \right]_j+ \Res_2
\ee
where
\be\label{residuals}
\Res_1 = \delta_1^+ \tilde{p}_j - \dot{\tilde{r}}_j
\mand \Res_2 = \sum_{m \ge 1} \delta_m^- V_m'\left( \G_m\tilde{r}\right)_j - \dot{\tilde{p}}_j.
\ee
The functions $\Res_1$ and $\Res_2$ are, like $R_\ep$, residuals and quantify the amount by which the approximators $\tilde{r}_j(t)$ and $\tilde{p}_j(t)$ fail to satisfy \eqref{CM3}. Ultimately these will be expressed in terms of $R_\ep$, but for now we leave things general.

Our goal in this section is to show that $\eta_j(t)$ and $\xi_j(t)$ remain small (in $\ell^2$) over long time periods, provided they are initially small. In particular we prove:
\begin{theorem}\label{gen approx}
Suppose that $\al>1$ with $2 \zeta_{\al +1} - \zeta_\al >0$. Assume further that for some $\tau_0,C_1, \ep_1>0$,
$\ep \in (0,\ep_1]$ implies
\be\label{res hyp}
\sup_{|t| \le \tau_0/ \ep^{\al}} \left(\|\Res_1\|_{\ell^2}  + \|\Res_2\|_{\ell^2}\right) \le C_1 \ep^{\beta},
\quad
\sup_{|t| \le \tau_0 /\ep^{\al}}\| \dot{\tilde{r}}\|_{\ell^\infty} \le C_1 \ep^{\al}
\ee
and
$$
\|\bar{\eta},\bar{\xi}\|_{\ell^2 \times \ell^2} \le C_1 \ep^{\beta-\al}.
$$
Then there exists constants $C_*, \ep_*>0$ so that the following holds for $\ep \in (0,\ep_*]$.
If $\eta_j(t)$, $\xi_j(t)$ solve \eqref{err eqn} with initial data $\bar{\eta}_j$, $\bar{\xi}_j$ we have
$$
\sup_{|t| \le \tau_0/\ep^{\al}} \| \eta(t),\xi(t)\|_{\ell^2 \times \ell^2} \le C_* \ep^{\beta - \al}.
$$

\end{theorem}

\begin{proof}

We begin by rewriting \eqref{err eqn} in a helpful way.
For $a,b \in \R$ put
$$
W_m(a,b) := V_m(b+a) -V_m(b)- V_m'(b)a 
$$
and let
$$
W_m'(a,b) := \partial_a W_m(a,b) = V'_m(b+a) - V'_m(b).
$$
With this \eqref{err eqn} becomes
\be\label{err eqn 2}
\dot{\eta}_j = \delta_1^+ \xi_j + \Res_1 \mand \dot{\xi}_j = \sum_{m \ge 1} \delta_m^- W_m'(\G_m \eta,\G_m \tilde{r}) _j+ \Res_2.
\ee
The point of introducing $W_m$ here is that now \eqref{err eqn 2} is structurally similar to \eqref{CM3} with $V_m$ replaced by $W_m$. Then we hope we can recapture some of the glory of conservation of $\E$ from above, but for the error equations. 

So for a solution of \eqref{err eqn 2} put:
$$
\H:={1 \over 2}\sum_{j \in \Z}  \xi_j^2 + \sum_{j \in \Z} \sum_{m\ge1} W_m(\G_m \eta,\G_m \tilde{r})_j.
$$
This is our replacement for $\E$. The following proposition contains the key properties of the second term in $\H$, chief of which that under some conditions it is equivalent to  $\|\eta\|_{\ell^2}^2$.
\begin{proposition} \label{W prop}
Fix $\al > 1$ with $2 \zeta_{\al+1} - \zeta_\al >0$. Then there exists $C>1$ such that the following hold when $\|\tilde{r}\|_{\ell^2} \le 1/4$ and $\|\eta\|_{\ell^2} \le 1/4$:
\be\label{W eq 1}
C^{-1} \|\eta\|_{\ell^2} \le \sqrt{ \sum_{j \in \Z} \sum_{m\ge1} W_m(\G_m \eta,\G_m \tilde{r})_j
} \le C  \|\eta\|_{\ell^2},
\ee
\be\label{W' est}
\sum_{m \ge 1} m \| W_m'(\G_m \eta,\G_m \tilde{r})\|_{\ell^2} \le C  \|\eta\|_{\ell^2},
\ee
\be\label{Wb est}
\sum_{m \ge 1} m \| \partial_b W_m(\G_m \eta,\G_m \tilde{r})\|_{\ell^1} \le C  \|\eta\|^2_{\ell^2}.
\ee
\end{proposition}
\begin{remark} Note that $W_m(a,0) = V_m(a)$ and so 
if $\tilde{r}_j(t)$ is identically zero then \eqref{W eq 1} coincides exactly with
the estimate in Lemma \ref{P equivalence}, with $\eta$ swapped with $r$.
\end{remark}

\begin{proof}
We begin with \eqref{W eq 1}. 
Taylor's theorem tells us that
$
W_m(a,b) = {\alpha (\alpha+1)  a^2/2 (m+b_*)^{\alpha+2}}
$
with $b_*$ in between $b$ and $b+a$. This leads to
\be\label{Wab eq}
{\alpha (\alpha+1) \over 2 (m+|a|+|b|)^{\alpha+2}} a^2\le W_m(a,b) \le {\alpha (\alpha+1) \over 2 (m-|a|
-|b|)^{\alpha+2}} a^2.
\ee

We have assumed $\|\eta\|_{\ell^2} \le 1/4$. 
Thus the classical estimate
 $\|f\|_{\ell^\infty} \le \|f\|_{\ell^2}$ along with the triangle inequality
 tell us
 $$
 |\G_m \eta_j| \le  \|\G_m \eta\|_{\ell^\infty} \le \sum_{l =0}^{m-1} \|\eta_{\cdot+l}\|_{\ell^\infty} = m \| \eta\|_{\ell^\infty}
 \le m/4.
 $$
 Similarly $\|\tilde{r}\|_{\ell^2} \le 1/4$ implies $ |\G_m \tilde{r}_j| \le m/4$. So we have
$ m - |\G_m \eta_j| - |\G_m \tilde{r}_j| \ge m/2$ 
and
$ m + |\G_m \eta_j| + |\G_m \tilde{r}_j| \ge 3m/2$ 
and thus \eqref{Wab eq} gives
 $$
 {2^{\alpha+1} \alpha (\alpha+1) \over 3^{\alpha+2} m^{\alpha+2}}(\G_m\eta_j)^2 \le W_m(\G_m\eta,\G_m\tilde{r})_j \le {2^{\alpha+1} \alpha (\alpha+1) \over  m^{\alpha+2}} (\G_m\eta_j)^2.
$$
Summing this gets us to
\bes\label{W eq 2}
 {2^{\alpha+1} \alpha (\alpha+1) \over 3^{\alpha+2} }P_2 \le 
 \sum_{j \in \Z} \sum_{m\ge1} W_m(\G_m \eta,\G_m \tilde{r})_j \le 
 2^{\alpha+1} \alpha (\alpha+1) P_2
\ees
where
$$
P_2:=\sum_{j \in \Z} \sum_{m \ge 1} {1 \over  m^{\alpha+2}} (\G_m \eta)_j ^2.
$$

Using the definition of $\G_m$ and multiplying out the square gives
$$
P_2=\sum_{j \in \Z} \sum_{m \ge 1} {1 \over  m^{\alpha+2}} \left(\sum_{l=0}^{m-1} \eta^2_{j+l} + 2 \sum_{0\le l <k \le m-1} \eta_{j+l} \eta_{j+k}\right).
$$
Rearranging sums and doing some computations on the ``diagonal part'' of the above gets
\bes\begin{split}
\sum_{j \in \Z} \sum_{m \ge 1} {1 \over  m^{\alpha+2}} \sum_{l=0}^{m-1} \eta^2_{j+l}
 =&\sum_{m \ge 1} {1 \over  m^{\alpha+2}} \sum_{l=0}^{m-1} \sum_{j \in \Z}\eta^2_{j+l}
  =\sum_{m\ge1} {1 \over  m^{\alpha+2}}\sum_{l=0}^{m-1} \| \eta\|_{\ell^2}^2
  = \zeta_{\alpha+1} \|\eta\|^2_{\ell^2}.
 \end{split}\ees
Therefore
 $$
 P_2-  \zeta_{\alpha+1} \|\eta\|^2_{\ell^2} =2\sum_{j \in \Z} \sum_{m \ge 1} {1 \over  m^{\alpha+2}} \sum_{0\le l <k \le m-1} \eta_{j+l} \eta_{j+k}=:P_{22}.
$$
Rearranging sums gives
$$
 P_{22}=2\sum_{m \ge 1} {1 \over  m^{\alpha+2}} \sum_{0\le l <k \le m-1} \sum_{j \in \Z} \eta_{j+l} \eta_{j+k}.
$$
We use Cauchy-Schwarz to get
$$
|P_{22}| \le 2 \sum_{m \ge 1} {1 \over  m^{\alpha+2}} \sum_{0\le l <k \le m-1} \|\eta\|_{\ell^2}^2.
$$
Since $\ds \sum_{0\le l <k \le m-1}1 =m(m-1)/2$, we obtain
$$
|P_{22}| \le  \|\eta\|_{\ell^2}^2 \sum_{m \ge 1} { m(m-1) \over   m^{\alpha+2}}
= \left( \zeta_{\alpha}-\zeta_{\alpha+1} \right)\|\eta\|_{\ell^2}^2 .
$$
Thus 
$$
\left\vert P_2-\zeta_{\alpha+1} \|\eta\|^2_{\ell^2}\right\vert \le \left( \zeta_{\alpha}-\zeta_{\alpha+1} \right)\|\eta\|_{\ell^2}^2 
$$
or rather
\bes\label{P2 equiv}
\left(2\zeta_{\alpha+1} -\zeta_\alpha\right)\|\eta\|^2_{\ell^2}  \le P_2 \le 
\zeta_{\alpha} \|\eta\|^2_{\ell^2} .
\ees
Therefore
$
2\zeta_{\alpha+1} -\zeta_\alpha > 0 
$
implies that $\sqrt{P_2}$ is equivalent to $\|\eta\|_{\ell^2}$.  

This in combination with \eqref{W eq 1} gives:
\bes\label{W eq 3}
 {2^{\alpha+1} \alpha (\alpha+1) \over 3^{\alpha+2} }\left(2\zeta_{\alpha+1} -\zeta_\alpha\right)\|\eta\|^2_{\ell^2} \le 
 \sum_{j \in \Z} \sum_{m\ge1} W_m(\G_m \eta,\G_m \tilde{r})_j \le 
 2^{\alpha+1} \alpha (\alpha+1) \zeta_{\alpha} \|\eta\|^2_{\ell^2}.
\ees
Which is to say we have \eqref{W eq 1}.

Next up is \eqref{W' est}.
The mean value theorem tell us that
$
W'_m(a,b) = {\alpha (\alpha+1)a/ (m+b_*)^{\alpha+2}}
$
where $b_*$ lies between $b$ and $b+a$
%
As above, we have 
 $\|\G_m \tilde{r}\|_{\ell^\infty} \le m/4$ and $\|\G_m \eta\|_{\ell^\infty} \le m/4$. So $b_*$ would be controlled above by $m/2$, for all $j$. Thus
$$
\left| W'_m(\G_m \eta,\G_m \tilde{r})_j \right| \le {2^{\alpha+2} \alpha(\alpha+1) \over m^{\alpha+2}} |\G_m \eta_j|.
$$
We take the $\ell^2$-norm and use the triangle inequality
$$
\left\| W'_m(\G_m \eta,\G_m \tilde{r})\right\|_{\ell^2} \le  {2^{\alpha+2} \alpha(\alpha+1) \over m^{\alpha+2}}
 \|\G_m \eta\|_{\ell^2} \le {2^{\alpha+2} \alpha(\alpha+1) \over m^{\alpha+1}}
 \|\eta\|_{\ell^2}
$$
Thus
$$
\sum_{m\ge1} m\left\| W'_m(\G_m \eta,\G_m \tilde{r})\right\|_{\ell^2} \le {2^{\alpha+2} \zeta_\al \alpha(\alpha+1)}
 \|\eta\|_{\ell^2}.
$$
This is \eqref{W' est}.

To get \eqref{Wb est} is more of the same. We have
$
\partial_b W_m(a,b) = -{\alpha(\alpha+1)(\alpha+2) a^2 /2 (m+b_*)^{\alpha+3}} 
$
with $b_*$ in between $b$ and $b+a$. Much as we did above, we get the estimate
$$
|\partial_b W_m(\G_m\eta,\G_m\tilde{r})_j| \le {2^{\alpha+3} \alpha(\alpha+1)(\alpha+2) \over m ^{\alpha+3}} 
|\G_m \eta_j|^2.
$$
Summing over $j$ and the triangle inequality lead to
$$
\|\partial_b W_m(\G_m\eta,\G_m\tilde{r})\|_{\ell^2} \le {2^{\alpha+3} \alpha(\alpha+1)(\alpha+2) \over m ^{\alpha+3}} 
\|\G_m \eta\|_{\ell^2}^2 \le  {2^{\alpha+3} \alpha(\alpha+1)(\alpha+2) \over m ^{\alpha+2}} 
\|\eta\|_{\ell^2}^2.
$$
Then
$$
\sum_{m \ge 1} m \|\partial_b W_m(\G_m\eta,\G_m\tilde{r})\|_{\ell^2}
\le  {2^{\alpha+3} \alpha(\alpha+1)(\alpha+2) \zeta_{\al+1}} 
\|\eta\|_{\ell^2}^2
$$
and we are done.
\end{proof}

With Proposition \ref{W prop} taken care of, we can now get into the energy argument at
the heart of the proof. We begin with
differentiation of $\H$ to get
$$
\dot{\H} = \sum_{j \in \Z}  \left(\xi_j \dot{\xi}_j + \sum_{m \ge 1}  W'_m(\G_m \eta,\G_m \tilde{r})_j \G_m \dot{\eta}_j\right)+\sum_{j \in \Z} \sum_{m\ge1} \partial_b W_m(\G_m \eta,\G_m \tilde{r})_j \G_m \dot{\tilde{r}}_j.
$$
Call the terms on the right $I$ and $II$ in the obvious way. Using \eqref{err eqn 2} in $I$ gets:
$$
I=\sum_{j \in \Z}  \left(\xi_j \left( \sum_{m \ge 1} \delta_m^- W_m'(\G_m \eta,\G_m \tilde{r}) _j+ \Res_2\right)+ \sum_{m \ge 1} W'_m(\G_m \eta,\G_m \tilde{r})_j\G_m(\delta_1^+ \xi_j + \Res_1 )\right).
$$
Summing by parts and using $\G_m \delta_1^+ = \delta_m^+$ gives some cancellations:
$$
I=\sum_{j \in \Z}  \left(\xi_j  \Res_2+ \sum_{m \ge 1} W'_m(\G_m \eta,\G_m \tilde{r})_j\G_m \Res_1 \right).
$$
Cauchy-Schwarz implies $|I| \le \|\xi\|_{\ell^2} \|\Res_2\|_{\ell_2} + |I_2|$
where
$$
I_2:=
 \sum_{m \ge 1}\sum_{j \in \Z} W'_m(\G_m \eta,\G_m \tilde{r})_j\G_m \Res_1 .
$$
The operator $\G_m$ is symmetric with respect to the $\ell^2$-inner product and so:
$$
I_2=
 \sum_{m \ge 1}\sum_{j \in \Z} \left(\G_m W'_m(\G_m \eta,\G_m \tilde{r})_j\right) \Res_1 .
$$
Reorder the sum again:
$$
I_2=
\sum_{j \in \Z}\Res_1   \sum_{m \ge 1} \left(\G_m W'_m(\G_m \eta,\G_m \tilde{r})_j\right). 
$$
Then Cauchy-Schwarz and the triangle inequality lead to
$$
|I_2|\le \|\Res_1\|_{\ell_2} \sum_{m \ge 1} \left\|\G_m W'_m(\G_m \eta,\G_m \tilde{r})
\right\|_{\ell^2}\le
\|\Res_1\|_{\ell^2}\sum_{m \ge 1} m \left\|   W'_m(\G_m \eta,\G_m \tilde{r})
\right\|_{\ell^2}.
$$
The estimate \eqref{W' est} from Proposition \ref{W prop} gives
$$
|I_2|  \le C \|\Res_1\|_{\ell^2} \|\eta\|_{\ell^2}.
$$
Thus we have
$$
|I| \le  C\left( \|\Res_1\|_{\ell^2} \|\eta\|_{\ell^2} + \|\Res_2\|_{\ell^2} \|\xi\|_{\ell^2} \right)
$$

Now look at $II$.  By using naive estimates we get
$$
|II| \le \sum_{j \in \Z} \sum_{m \ge 1} |\partial_b W_m(\G_m\eta,\G_m\tilde{r})_j| \|\G_m \dot{\tilde{r}}\|_{\ell^\infty}
\le\| \dot{\tilde{r}}\|_{\ell^\infty} \sum_{m \ge 1} m  \|\partial_b W_m(\G_m\eta,\G_m\tilde{r})\|_{\ell^1}.
$$
Then \eqref{Wb est} from Proposition \ref{W prop} yields
\bes
\begin{split}
|II| \le
C\| \dot{\tilde{r}}\|_{\ell^\infty} \| \eta\|_{\ell^2}^2.
\end{split}
\ees

So all together 
$$
\dot{\H} \le   C\left( \|\Res_1\|_{\ell^2} \|\eta\|_{\ell^2} + \|\Res_2\|_{\ell^2} \|\xi\|_{\ell^2} \right)
+C\| \dot{\tilde{r}}\|_{\ell^\infty} \| \eta\|_{\ell^2}^2.
$$
Using \eqref{W eq 1} we have:
$$
\dot{\H} \le   C\left( \|\Res_1\|_{\ell^2}  + \|\Res_2\|_{\ell^2}\right)\sqrt{\H}
+C\| \dot{\tilde{r}}\|_{\ell^\infty} \H.
$$
The assumptions made on $\Res_1$, $\Res_2$ and $\dot{\tilde{r}}$ lead to
$$
\dot{\H} \le   C\ep^\beta\sqrt{\H}
+C\ep^\al \H.
$$
Applying Gr\"onwall's inequality yields
$$
\sqrt{\H(t)} \le \ep^{C \ep^\al t} \sqrt{\H(0)} + C \ep^{\beta - \al} \left(\ep^{C \ep^\al t}-1 \right).
$$
Then we use \eqref{W eq 1} one last time to get
$$
\|\eta(t),\xi(t)\|_{\ell^2\times \ell^2} \le \ep^{C \ep^\al t} \|\bar{\eta},\bar{\xi}\|_{\ell^2\times \ell^2}+ C \ep^{\beta - \al} \left(\ep^{C \ep^\al t}-1 \right).
$$
Taking the supremum over $|t| \le \tau_0/\ep^\al$ and using the assumption on the size of the initial data gives the final estimate in the theorem.
\end{proof}

\section{Proof of Theorem \ref{main}}\label{main thm section}
The proof of Theorem \ref{main} is more or less a direct application of the general approximation theorem, Theorem \ref{gen approx}. There are a few small details to attend to, and that is what we do now.

\begin{proof} (Theorem \ref{main}).
Fix $u(X,\tau)$ a solution of \eqref{BO} subject to the bound described in the statement of the Theorem.
Let $\ds
v(X,\tau) := -\int_0^X u(b,\tau) d b
$ so that $u = -\partial_X v$.
Form $\tilde{x}_j$ as in \eqref{ansatz},
namely $\tilde{x}_j(t) := j + \ep^{\al -2} v(\ep(j-c_\alpha t),\ep^{\al} t)$.
Then put
$$
\tilde{r}_j(t) := -1+\delta_1^+ \tilde{x}_j(t) \mand \tilde{p}_j(t) := \dot{\tilde{x}}_j(t).
$$

We first compute $\Res_1$ and $\Res_2$ as in \eqref{residuals}. 
We have
$$
\Res_1 = \delta_1^+ \tilde{p} - \dot{\tilde{r}}=
\delta_1^+ \left( \dot{\tilde{x}}\right) - \partial_t\left({ -1+\delta_1^+ \tilde{x} }\right) = 0.
$$
For $\Res_2$, we compute
\bes\begin{split}
\Res_2 
&= \sum_{m \ge 1} \delta_m^- V_m'\left( \G_m\tilde{r}\right)_j - \dot{\tilde{p}}_j\\
&= \sum_{m \ge 1} \delta_m^- V_m'\left( -m+\delta_m^+ \tilde{x})\right)_j - \ddot{\tilde{x}}_j\\
& = -\alpha\sum_{m \ge 1} \left( {1 \over (\tilde{x}_{j+m} - \tilde{x}_j)^{\alpha+1}} -{1 \over (\tilde{x}_{j} - \tilde{x}_{j-m})^{\alpha+1}} 
  \right)- \ddot{\tilde{x}}_j\\
  &=-R_\ep
\end{split}
\ees
with $R_\ep$ as above. Thus Proposition \ref{Res est 1} tells us that the hypothesis on the residuals, \eqref{res hyp}, in Theorem
\ref{gen approx} is met with $\beta = \beta_\al$. Note that $\gamma_\al = \beta_\al-\al$.

The fundamental theorem of calculus gives
\be\label{lw A}
\delta_1^+ (f(\ep \cdot))_j =
f(\ep(j+1))-f(\ep j) = \int_{\ep j}^{\ep j + \ep} f_X(X) dX =  \ep (A_\ep f_X)(\ep j).
\ee
If we use this and the relation $u = -\partial_X v$ we have
\bes\begin{split}
\tilde{r}_j(t) &= -\ep^{\al-1} (A_\ep u)(\ep(j-c_\al t),\ep^{\al} t)\\&=
-\ep^{\al-1} u(\ep(j-c_\al t),\ep^{\al} t)
-\ep^{\al-1} ((A_\ep -1)u)(\ep(j-c_\al t),\ep^{\al} t).
\end{split}\ees
So \eqref{lw sample} and the final estimate in Lemma \ref{A est} gives us
\be\label{tr est}
\sup_{|t| \le \tau_0 /\ep^\al} \|\tilde{r}(t) +\ep^{\al-1} u(\ep(\cdot-c_\al t),\ep^{\al} t)\|_{\ell^2} \le 
C\ep^{\al-1/2}.
\ee
The assumption on the initial conditions in Theorem \ref{main} 
implies
$$
\|r(0) + \ep^{\al-1}u(\ep \cdot,0)\| =\|\bar{\mu}\|_{\ell^2} \le C\ep^{\beta_\al - \al}.
$$
It is easy enough to check that $\ep^{\al -1/2} \le \ep^{\beta_\al - \al}$  and therefore \eqref{tr est}
and the triangle inequality cough up
$$
\|r(0) - \tilde{r}(0)\|_{\ell^2} \le C \ep^{\beta_\al- \al}
$$
which
is one of the hypotheses on the initial data in Theorem \ref{gen approx}. 

Similarly we have
\bes\begin{split}
\tilde{p}_j(t) & = c_\al \ep^{\al -1} u(\ep(j-c_\al t),\ep^\al t) + \ep^{2\al -2} v_\tau(\ep(j-c_\al t),\ep^\al t).
\end{split}\ees
It is straightforward to use \eqref{BO}, the relation $u = -\partial_X v$ and \eqref{lw sample}
to show $\|v_\tau(\ep(\cdot-c_\al t),\ep^\al t)\|_{\ell^2} \le C \ep^{-1/2}$ and so 
\be\label{tp est}
\sup_{|t| \le \tau_0 /\ep^\al} \|\tilde{p}(t) - c_\al \ep^{\al-1} u(\ep(\cdot-c_\al t),\ep^{\al} t)\|_{\ell^2} \le 
C\ep^{2\al-3/2} \le C \ep^{\beta_\al - \al}.
\ee
The assumption on the initial conditions in Theorem \ref{main}  tell us:
$$
\|p(0) -c_\al \ep^{\al-1}u(\ep \cdot,0)\| =\|\bar{\nu}\|_{\ell^2} \le C\ep^{\beta_\al - \al}.
$$
Therefore
$$
\|p(0) - \tilde{p}(0)\|_{\ell^2} \le C \ep^{\beta_\al- \al}
$$
which
is the other hypothesis on the initial data in Theorem \ref{gen approx}.

Next, since,
$
\dot{\tilde{r}}_j(t) = \delta_1^+ \tilde{p}_j(t),
$
\eqref{lw A} and $u = -\partial_X v$ give us
$$
\dot{\tilde{r}}_j(t) = -c_\al \ep^\al (A_{\ep} u_X) (\ep (j - c_\al t),\ep^\al t)  
-\ep^{2 \al -1} (A_{\ep} u_\tau) (\ep (j - c_\al t),\ep^\al t). 
$$
An easy estimate provides
$$
\|\dot{\tilde{r}}(t)\|_{\ell^\infty} \le  c_\al \ep^\al \|A_{\ep} u_X (\cdot,\ep^\al t)\|_{L^\infty}
+\ep^{2 \al -1} \|A_{\ep} u_\tau (\cdot,\ep^\al t)\|_{L^\infty} 
$$
Using Sobolev, followed by the first estimate in Lemma \ref{A est}:
$$
\|\dot{\tilde{r}}(t)\|_{\ell^\infty} \le C \ep^\al \|u (\cdot,\ep^\al t)\|_{H^2}
+C\ep^{2 \al -1} \| u_\tau (\cdot,\ep^\al t)\|_{H^1}.
$$
Then we use \eqref{BO} and the uniform bound on $u$ to get
$$
\sup_{|t|\le \tau_0/\ep^\al}\|\dot{\tilde{r}}(t)\|_{\ell^\infty} \le C \ep^\al + C\ep^{2 \al -1} \le C \ep^{\al}.
$$
This gives the estimate on $\dot{\tilde{r}}$ in \eqref{res hyp}.

We have now checked off all the hypotheses of Theorem \ref{gen approx} and thus its conclusions hold. And so we find that
$$
\sup_{|t| \le \tau/\ep^\al}\| r(t) - \tilde{r}(t)\|_{\ell^2} \le C \ep^{\beta_\al - \al}.
$$
This, together with \eqref{tr est} and the triangle inequality give:
$$
\sup_{|t| \le \tau/\ep^\al}\|r(t) + \ep^{\al -1} u(\ep(\cdot -t),\ep^\al t)\|_{\ell^2} \le C\ep^{\beta_\al - \al}
$$
which is the estimate on $\mu(t)$ in Theorem \ref{main}. The estimate on $\nu(t)$ follows
from $$
\sup_{|t| \le \tau/\ep^\al}\| p(t) - \tilde{p}(t)\|_{\ell^2} \le C \ep^{\beta_\al - \al}.
$$
and \eqref{tp est} in the same way.

\end{proof}

\bibliographystyle{plain}
\bibliography{bo-validation}{}

\begin{thebibliography}{1}

\bibitem{BLC}
J.~L. Bona, D.~Lannes, and J.-C. Saut.
\newblock Asymptotic models for internal waves.
\newblock {\em J. Math. Pures Appl. (9)}, 89(6):538--566, 2008.

\bibitem{GMWZ}
Jeremy Gaison, Shari Moskow, J.~Douglas Wright, and Qimin Zhang.
\newblock Approximation of polyatomic {FPU} lattices by {K}d{V} equations.
\newblock {\em Multiscale Model. Simul.}, 12(3):953--995, 2014.

\bibitem{HML}
Michael Herrmann and Alice Mikikits-Leitner.
\newblock Kd{V} waves in atomic chains with nonlocal interactions.
\newblock {\em Discrete and Continuous Dynamical Systems}, 36(4):2047--2067.

\bibitem{hur}
Vera~Mikyoung Hur.
\newblock Norm inflation for equations of {K}d{V} type with fractional
  dispersion.
\newblock {\em Differential Integral Equations}, 31(11-12):833--850, 2018.

\bibitem{IRTW}
Mihaela Ifrim, James Rowan, Daniel Tataru, and Lizhe Wan.
\newblock The {B}enjamin-{O}no approximation for 2{D} gravity water waves with
  constant vorticity.
\newblock {\em Ars Inven. Anal.}, pages Paper No. 3, 33, 2022.

\bibitem{IP}
Benjamin Ingimarson and Robert~L. Pego.
\newblock On long waves and solitons in particle lattices with forces of
  infinite range, 2023.

\bibitem{matsuno}
Yoshimasa Matsuno.
\newblock {\em Bilinear Transformation Method}.
\newblock Academic Press, London, 1984.

\bibitem{SW}
Guido Schneider and C.~Eugene Wayne.
\newblock Counter-propagating waves on fluid surfaces and the continuum limit
  of the {F}ermi-{P}asta-{U}lam model.
\newblock In {\em International {C}onference on {D}ifferential {E}quations,
  {V}ol. 1, 2 ({B}erlin, 1999)}, pages 390--404. World Sci. Publ., River Edge,
  NJ, 2000.

\end{thebibliography}
\end{document}